\documentclass[12pt]{article}
\usepackage{tikz}

\usepackage{amsmath}
\usepackage{amsfonts}
\usepackage{amssymb}

\newtheorem{definition}{Definition}
\newtheorem{remark}{Remark}

\newtheorem{theorem}{Theorem}

\newtheorem{proposition}{Proposition}

\newtheorem{problem}{Problem}

\newenvironment{proof}{\textbf{Proof} :}{$\square$}

\title{Potential Methods for Extending Galvin and J\'onsson's Characterization of Distributive Sublattices of Free Lattices}
\author{Brian T. Chan}
\date{Department of Mathematics and Statistics, The University of Calgary}

\begin{document}

\maketitle

\begin{abstract}
In 1959, F.Galvin and B.J\'onsson characterized distributive sublattices of free lattices in their paper \cite{DSFL}. In this paper, I will create new proofs to a portion of Galvin and J\'onsson's work in \cite{DSFL}. Based on these new proofs, I will explore possible generalizations of F.Galvin and B.J\'onsson's work by defining \emph{spanning pairs} and proving partial results which may help with analysing finite width sublattices of free lattices; and by making some new observations on finitely generated lattices over semidistributive varieties.

The work done in this paper may assist in attacking the following long-standing open problem: Which countable lattices are isomorphic to a sublattice of a free lattice?
\end{abstract}

\section{Distributive Sublattices of Free Lattices}

Among lattices known to be isomorphic to countable sublattices of free lattices, those isomorphic to distributive sublattices of free lattices are the easiest to describe. In 1959, F. Galvin and B. J\'onsson characterized, up to isomorphism, distributive sublattices of free lattices in their paper \cite{DSFL}. Their characterization will be over-viewed in this section and I will derive new proofs to a portion of their work. These new proofs have some similarities with work in \cite{DSFL} from F. Galvin and B. J\'onsson and with work from 1974 in \cite{OLE} by W. Poguntke and I. Rival.\\

We first introduce some basic operations. If $K$ and $L$ are lattices, then $K \times L$ is their direct product where $(p,q) \vee (r,s) = (p \vee r, q \vee s)$ and $(p,q) \wedge (r,s) = (p \wedge r, q \wedge s)$ for all $p \in K$ and $q \in L$. Let $\langle L ; \leq_L \rangle$ and $\langle K ; \leq_K \rangle$ be lattices such that $L \cap K = \varnothing$. Then the \emph{linear sum} $L \oplus K$ of $L$ and $K$ is the lattice $\langle L \cup K ; \leq \rangle$ such that: \\

(1) For all $p,q \in L$, $p \leq q$ if and only if $p \leq_L q$.\\

(2) For all $p,q \in K$, $p \leq q$ if and only if $p \leq_K q$.\\

(3) For all $p \in L$ and $q \in K$, $p \leq q$.\\

Generalizing this operation, one may consider a family of lattices $\{ \langle L_i ; \leq_i \rangle : i \in I \}$ indexed by a partially ordered set $\langle I ; \leq_I \rangle$ such that for all indices $i \neq j$ in $I$, $L_i \cap L_j = \varnothing$. Then the \emph{lexicographic sum} of this family is the lattice $\langle \bigcup_{i \in I} L_i ; \leq \rangle$ with the binary relation $\leq$ defined as follows: \\

(1) For all $p,q \in L_i$, $p \leq q$.\\

(2) For all $i \leq_I j$ in $\langle I ; \leq_I \rangle$, if $i \neq j$, then for all $p \in L_i$ and $q \in L_j$, $p \leq q$.\\

Inspired by \cite{ILO} and \cite{DSFL} we will define, in this paper, a \emph{linear sum indexed by $I$} to be the lexicographic sum of the family $\{ \langle L_i ; \leq_i \rangle : i \in I \}$ indexed by a chain $I$. We denote such a lexicographic sum by $\bigoplus_{i \in I}L_i$.\\

F. Galvin and B. J\'onsson's characterization of lattices isomorphic to distributive sublattices of free lattices is as follows. 
We use notation from \cite{ILO} by letting $\textbf{n} = \langle \{0,1, \dots, n-1 \} ; \leq \rangle$ with $0 < 1 < \dots < n-1$:\\

\begin{theorem}[F. Galvin and B. J\'onsson] \label{the characterization}
	Let $L$ be a distributive lattice. Then $L$ is isomorphic to a sublattice of a free lattice if and only if $L \cong \bigoplus_{i \in I} L_i$ where $\bigoplus_{i \in I} L_i$ is a linear sum indexed by a countable chain $I$ and for all $i \in I$: $|L_i| = 1$, $L_i \cong \textbf{2} \times \textbf{2} \times \textbf{2}$, or $L = \textbf{2} \times C$ where $C$ is a countable chain.
\end{theorem}

Theorem 3 implies that all distributive sublattices of free lattices are countable and have a width of at most three. An important ingredient for Theorem \ref{the characterization} is Theorem \ref{the restriction}:

\begin{theorem}[F. Galvin and B. J\'onsson, \cite{DSFL}] \label{the restriction}
	Let $L$ be a distributive lattice. Then $L$ has no doubly reducible elements if and only if $L \cong \bigoplus_{i \in I} L_i$ where $\bigoplus_{i \in I} L_i$ is a linear sum indexed by a chain $I$ and for all $i \in I$: $|L_i| = 1$, $L_i \cong \textbf{2} \times \textbf{2} \times \textbf{2}$, or $L = \textbf{2} \times C$ where $C$ is a chain.
\end{theorem}

Theorem \ref{the restriction} characterizes distributive lattices which satisfy Whitman's Condition. This is because all lattices that satisfy Whitman's Condition have no doubly reducible elements and all lattices described in Theorem \ref{the restriction} satisfy Whitman's Condition. \\

To prove Theorem \ref{the restriction} three propositions are needed: Proposition \ref{powerful}, Proposition \ref{two two two}, and Proposition \ref{ladder}. These three propositions constitute the majority of Galvin and J\'onsson's paper \cite{DSFL}. Proposition \ref{powerful} is quite well-known in literature on free lattices. We refer the reader to \cite{FL}, \cite{LTF}, or to \cite{DSFL} to see how to prove Proposition \ref{powerful}. In the next section, I will write my new proofs to Proposition \ref{two two two} and Proposition \ref{ladder}.

\begin{proposition}[Galvin and J\'onsson] \label{powerful}
	All chains in a free lattice are countable.
\end{proposition}

\begin{proposition}[Galvin and J\'onsson] \label{two two two}
	If $L$ is a linearly indecomposable distributive lattice with no doubly reducible elements, then $L$ has a width of at most two or $L$ is isomorphic to $\textbf{2} \times \textbf{2} \times \textbf{2}$.
\end{proposition}

\begin{proposition}[Galvin and J\'onsson] \label{ladder}
	All width two distributive lattices that are linearly indecomposable and have no doubly reducible elements are isomorphic to $\textbf{2} \times C$ where $C$ is a chain.
\end{proposition}

\subsection{Distributive Sublattices Part 1}

We analyse distributive sublattices of free lattices which have a width of at least three. Specifically, Proposition \ref{two two two} from F. Galvin and B. J\'onsson is proven in a new way. F. Galvin and B. J\'onsson's original proof involved intricate calculations with carefully selected lattice terms (see \cite{DSFL}). Although the new proof in this subsection has similaritie s to F. Galvin and B. J\'onsson's arguments, it uses a lattice, the \emph{free distributive lattice on 3 generators}, to capture all possible cases that may be encountered.\\

The following more general concept (see \cite{FL} and \cite{VL}) is useful for describing free distributive lattices:

\begin{definition}(Relatively Free Lattices) \label{relatively free}
	Let $\mathcal{V}$ be a variety of lattices and $X$ be a set. Then the relatively free lattice in $\mathcal{V}$ over $X$ is the lattice unique up to lattice isomorphism, $F_\mathcal{V}(X)$, such that for any lattice $L$ in the variety $\mathcal{V}$ and any mapping, $f : X \to L$, there is a lattice homomorphism, $h : F_\mathcal{V}(X) \to L$, such that, for all $x \in X$, $h(x) = f(x)$.
\end{definition}

A \emph{free distributive lattice} on a set $X$, written as $FD(X)$, is the relatively free lattice $F_{\mathcal{D}}(X)$ where $\mathcal{D}$ is the variety of distributive lattices. If $\mathcal{V}$ is a variety of lattices, we sometimes denote $F_\mathcal{V}(X)$ by $F_\mathcal{V}(|X|)$. The elements of $X$ are called the \emph{generators} of $F_\mathcal{V}(X)$.  We occasionally denote $FD(X)$ by $FD(|X|)$.\\

My proof of Proposition \ref{two two two} is as follows.\\

\begin{proof} (Brian T. Chan) Let $D$ be a distributive lattice with no doubly reducible elements. Instead of manipulating terms directly in $D$, as done in \cite{DSFL}, we look at the \emph{free distributive lattice on 3 generators} in Figure 3.1 which we denote by $FD(\{a,b,c\})$. It can be shown (see \cite{ILO}) that for any positive integer $n$, the elements of $FD(n)$ can be identified with the nonempty antichains of $\mathcal{P}(\mathcal{P}(S) \backslash \{\varnothing\})$ where $\mathcal{P}$ is the power set operation and $|S| = n$. \\

We first prove Lemma 1 of \cite{DSFL}: \emph{If $\{p,q,r\} \subseteq D$ is an antichain, then the sublattice generated by $\{p,q,r\}$ is isomorphic to $\textbf{2} \times \textbf{2} \times \textbf{2}$.} \\

As $D$ is distributive, there is a lattice homomorphism $f : FD(\{a,b,c\}) \rightarrow D$ induced by the mapping: $a \mapsto p$, $b \mapsto q$, and $c \mapsto r$. Now let $\equiv$ be the congruence on $FD(\{a,b,c\}$ such that $s \equiv t$ if and only if $f(s) = f(t)$. Consider the element $z = (a \vee b) \wedge (b \vee c) \wedge (c \vee a) = (a \wedge b) \vee (b \wedge c) \vee (c \wedge a)$ in $FD(\{a,b,c\})$. \\

In \cite{DSFL}, the element $f(z)$ was analysed in $D$. Moreover, the authors assumed without loss of generality that $f(z)$ is join irreducible. We use this assumption for $FD(\{a,b,c\})$ by assuming that $z$ is join irreducible. For this to occur we assume without loss of generality that $z \equiv (a \wedge b) \vee (a \wedge c)$ and $z \equiv (a \wedge c) \vee (b \wedge c)$. \\

Since congruences are \emph{compatible with join and meet}, we can more easily calculate quotient lattices of a finite lattice. We recall the notion of a \emph{congruence lattice} (see \cite{LTF}): In Figure 3.2, the leftmost lattice is the quotient lattice $FD(\{a,b,c\})/\equiv_1$ where $\equiv_1$ is the congruence $con(x, (a \wedge b) \vee (a \wedge c)) \in Con(FD(\{a,b,c\}))$ and the middle lattice is the quotient lattice $FD(\{a,b,c\})/\equiv_2$ where $\equiv_2$ is the congruence \\

$con(x, (a \wedge b) \vee (a \wedge c))$ $\vee$ $con(x, (a \wedge c) \vee (b \wedge c)) \in Con(FD(\{a,b,c\}))$. \\

Referring to Figure 3.2, consider $y = b \vee (a \wedge c) = (a \vee b) \wedge (a \vee c)$ in $FD(\{a,b,c\})/\equiv$. Another parallel between this proof and Galvin and J\'onsson's original proof is that in \cite{DSFL}, the element $f(y) \in D$ was analysed. Since $\{a,b,c\}$ is an antichain in $FD(\{a,b,c\})$ and in $D$, $\{a,b,c\}$ is an antichain in $FD(\{a,b,c\})/\equiv_2$. Moreover, $f(y)$ cannot be doubly reducible in $D$ and lattice congruences are compatible with join and meet. Hence, $y \equiv b$ and the resulting quotient lattice, $FD(\{a,b,c\})/\equiv$, is obtained thus proving Lemma 1 of \cite{DSFL}. \\

\begin{remark} \label{pool}
In any distributive lattice $H$ if $\{p_1,p_2, \dots, p_n\} \subseteq H$ is an antichain such that for some $z \in H$, $z = p_i \wedge p_j$ for all $i \neq j$ then it can be verified using the distributive laws that the sublattice of $H$ generated by $\{p_1,p_2, \dots, p_n\}$ is isomorphic to the $2^n$ element boolean algebra (see \cite{DSFL}).
\end{remark}

\begin{remark} \label{forbidden}
Since $D$ has no doubly reducible elements, we note that $D$ cannot have a sublattice isomorphic to $\textbf{2} \times \textbf{2} \times \textbf{2} \times \textbf{2}$ or to $\textbf{2} \times \textbf{2} \times \textbf{3}$ because those two lattices have doubly reducible elements.
\end{remark}

To see that the width of $D$ is three, suppose that $\{p,q,r,s\} \subseteq D$ is an antichain. We use Lemma 1 of \cite{DSFL}. The lattice $\textbf{2} \times \textbf{2} \times \textbf{2}$ has two antichains of cardinality three: $\{(1,0,0),(0,1,0),(0,0,1) \}$ and $\{(0,1,1), (1,0,1), (1,1,0) \}$. So assume without loss of generality that $p \wedge q = q \wedge r = r \wedge p$. If $q \wedge r = r \wedge s = s \wedge q$ then by Remark \ref{pool}, $\{p,q,r,s\}$ generates a sublattice of $D$ isomorphic to the boolean lattice $\textbf{2} \times \textbf{2} \times \textbf{2} \times \textbf{2}$. Moreover, if $q \vee r = r \vee s = s \vee q$ then as the sublattice of $D$ generated by $\{q,r,s\}$ is isomorphic to $\textbf{2} \times \textbf{2} \times \textbf{2}$, it can be checked that $\{p,q,r,s\}$ generates a sublattice of $D$ isomorphic to $\textbf{2} \times \textbf{2} \times \textbf{3}$. But that is impossible by Remark \ref{forbidden}. Hence, the width of $D$ must be three. \\

Since the width of $D$ is three, let $\{p,q,r\} \subseteq D$ be an antichain and assume without loss of generality that for some element $z \in D$, $z = p \wedge q = q \wedge r = r \wedge p$. Because $D$ is linearly indecomposable, it is enough to show by contradiction that if $s \in D$ is not in the sublattice generated by $\{p,q,r\}$ then $s \leq z = p \wedge q \wedge r$ or $s \geq p \vee q \vee r$. \\

So suppose that the above claim in not true. By duality it is enough to consider when $s$ is comparable to exactly one element of $\{p,q,r\}$ and less than or equal to at least one element of $\{p \vee q, q \vee r, r \vee p \}$. So without loss of generality we can assume that $z < s < p$ for the following reasons: \\

(1) \emph{If $s > p$ we can use the fact that $D$ is modular and consider the element $s \wedge p$. Similar reasoning applies to when $s > q$ or $s > r$.}\\

(2) \emph{If $s < p$ we can use the fact that $D$ has no doubly reducible elements and consider the elements $s \vee z$ and $p = (p \vee q) \wedge (p \vee r)$. Similar reasoning applies to when $s < q$ or $s < r$.}\\

We simplify an argument from \cite{DSFL} as follows. Because $z < s < p$ and $p \wedge q = p \wedge r = z$, $s \parallel q$ and $s \parallel r$; so $\{s,q,r\}$ is an antichain. Moreover it can be checked, using the fact that $D$ is modular, that $\{s, s \vee q, s \vee r\}$ is also an antichain. Hence, it can be seen that by Lemma 1 of \cite{DSFL}, the sublattice of $D$ generated by $\{p,q,r,s\}$ is isomorphic to $\textbf{2} \times \textbf{3}$. But that is impossible by Remark \ref{forbidden}. This completes the proof.

	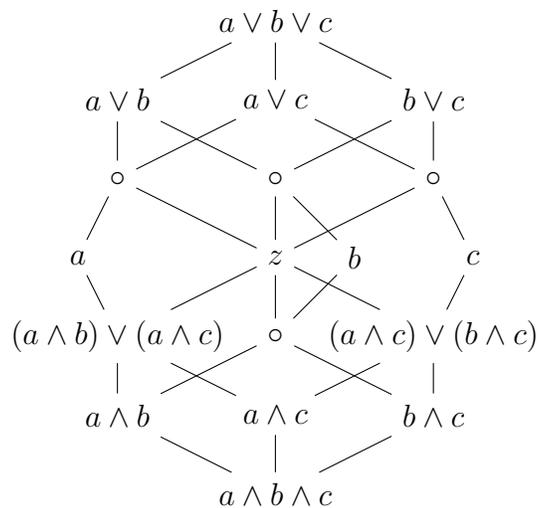
\begin{figure}
		
		\begin{center}
			\begin{tikzpicture}[scale = 2.1]
			
			\node (1) at (0,0) {$a \vee b \vee c$};
			
			\node (a+b) at (-1,-0.5) {$a \vee b$};
			\node (a+c) at (0,-0.5) {$a \vee c$};
			\node (b+c) at (1,-0.5) {$b \vee c$};
			
			\node (a+b-) at (-1,-1) {$\circ$};
			\node (a+c-) at (0,-1) {$\circ$};
			\node (b+c-) at (1,-1) {$\circ$};
			
			\node (z) at (0,-1.5) {$z$};
			
			\node (ab+) at (-1,-2) {$(a \wedge b) \vee (a \wedge c)$};
			\node (ac+) at (0,-2) {$\circ$};
			\node (bc+) at (1,-2) {$(a \wedge c) \vee (b \wedge c)$};
			
			\node (ab) at (-1,-2.5) {$a \wedge b$};
			\node (ac) at (0,-2.5) {$a \wedge c$};
			\node (bc) at (1,-2.5) {$b \wedge c$};
			
			\node (0) at (0,-3) {$a \wedge b \wedge c$};
			
			\node (a) at (-1.25,-1.5) {$a$};
			\node (b) at (0.5,-1.5) {$b$};
			\node (c) at (1.25,-1.5) {$c$};
			
			\draw (1) -- (a+b);
			\draw (1) -- (a+c);
			\draw (1) -- (b+c);
			
			\draw (z) -- (a+b-);
			\draw (z) -- (a+c-);
			\draw (z) -- (b+c-);
			
			\draw (z) -- (ab+);
			\draw (z) -- (ac+);
			\draw (z) -- (bc+);
			
			\draw (0) -- (ab);
			\draw (0) -- (ac);
			\draw (0) -- (bc);
			
			\draw (a+b) -- (a+b-);
			\draw (b+c) -- (b+c-);
			\draw (a+b) -- (a+c-);
			\draw (a+b-) -- (a+c);
			\draw (a+c-) -- (b+c);
			\draw (a+c) -- (b+c-);
			
			\draw (ab) -- (ab+);
			\draw (bc) -- (bc+);
			\draw (ab) -- (ac+);
			\draw (ab+) -- (ac);
			\draw (ac+) -- (bc);
			\draw (ac) -- (bc+);
			
			\draw (ab+) -- (a) -- (a+b-);
			\draw (ac+) -- (b) -- (a+c-);
			\draw (bc+) -- (c) -- (b+c-);
			
			\end{tikzpicture}
		\end{center}
		\caption{The lattice $FD(\{a,b,c\})$}
	\end{figure}
	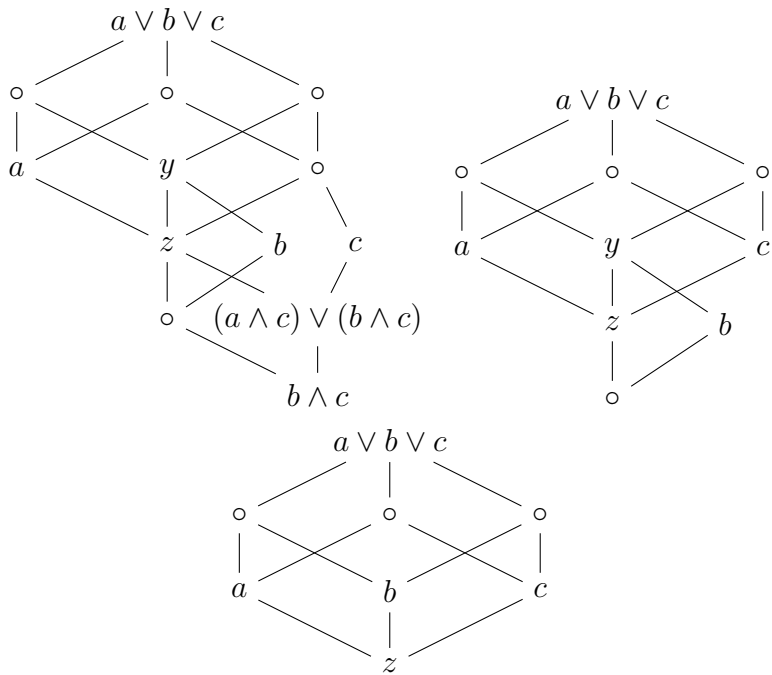
\begin{figure}
		
		\begin{center}
			\begin{tikzpicture}[scale = 2]
			
			\node (1) at (0,0) {$a \vee b \vee c$};
			
			\node (a+b) at (-1,-0.5) {$\circ$};
			\node (a+c) at (0,-0.5) {$\circ$};
			\node (b+c) at (1,-0.5) {$\circ$};
			
			\node (a+b-) at (-1,-1) {$a$};
			\node (a+c-) at (0,-1) {$y$};
			\node (b+c-) at (1,-1) {$\circ$};
			
			\node (z) at (0,-1.5) {$z$};
			
			\node (ac+) at (0,-2) {$\circ$};
			\node (bc+) at (1,-2) {$(a \wedge c) \vee (b \wedge c)$};
			
			\node (bc) at (1,-2.5) {$b \wedge c$};
			
			\node (b) at (0.75,-1.5) {$b$};
			\node (c) at (1.25,-1.5) {$c$};
			
			\draw (1) -- (a+b);
			\draw (1) -- (a+c);
			\draw (1) -- (b+c);
			
			\draw (z) -- (a+b-);
			\draw (z) -- (a+c-);
			\draw (z) -- (b+c-);
			
			\draw (z) -- (ac+);
			\draw (z) -- (bc+);
			
			\draw (a+b) -- (a+b-);
			\draw (b+c) -- (b+c-);
			\draw (a+b) -- (a+c-);
			\draw (a+b-) -- (a+c);
			\draw (a+c-) -- (b+c);
			\draw (a+c) -- (b+c-);
			
			\draw (bc) -- (bc+);
			\draw (ac+) -- (bc);
			
			\draw (ac+) -- (b) -- (a+c-);
			\draw (bc+) -- (c) -- (b+c-);
			
			\end{tikzpicture}
			\begin{tikzpicture}[scale = 2]
			
			\node (1) at (0,0) {$a \vee b \vee c$};
			
			\node (a+b) at (-1,-0.5) {$\circ$};
			\node (a+c) at (0,-0.5) {$\circ$};
			\node (b+c) at (1,-0.5) {$\circ$};
			
			\node (a+b-) at (-1,-1) {$a$};
			\node (a+c-) at (0,-1) {$y$};
			\node (b+c-) at (1,-1) {$c$};
			
			\node (z) at (0,-1.5) {$z$};
			
			\node (ac+) at (0,-2) {$\circ$};
			
			\node (b) at (0.75,-1.5) {$b$};
			
			\draw (1) -- (a+b);
			\draw (1) -- (a+c);
			\draw (1) -- (b+c);
			
			\draw (z) -- (a+b-);
			\draw (z) -- (a+c-);
			\draw (z) -- (b+c-);
			
			\draw (z) -- (ac+);
			
			\draw (a+b) -- (a+b-);
			\draw (b+c) -- (b+c-);
			\draw (a+b) -- (a+c-);
			\draw (a+b-) -- (a+c);
			\draw (a+c-) -- (b+c);
			\draw (a+c) -- (b+c-);
			
			\draw (ac+) -- (b) -- (a+c-);
			
			\end{tikzpicture}
			\begin{tikzpicture}[scale = 2]
			
			\node (1) at (0,0) {$a \vee b \vee c$};
			
			\node (a+b) at (-1,-0.5) {$\circ$};
			\node (a+c) at (0,-0.5) {$\circ$};
			\node (b+c) at (1,-0.5) {$\circ$};
			
			\node (a+b-) at (-1,-1) {$a$};
			\node (a+c-) at (0,-1) {$b$};
			\node (b+c-) at (1,-1) {$c$};
			
			\node (z) at (0,-1.5) {$z$};
			
			\draw (1) -- (a+b);
			\draw (1) -- (a+c);
			\draw (1) -- (b+c);
			
			\draw (z) -- (a+b-);
			\draw (z) -- (a+c-);
			\draw (z) -- (b+c-);
			
			\draw (a+b) -- (a+b-);
			\draw (b+c) -- (b+c-);
			\draw (a+b) -- (a+c-);
			\draw (a+b-) -- (a+c);
			\draw (a+c-) -- (b+c);
			\draw (a+c) -- (b+c-);
			
			\end{tikzpicture}
		\end{center}
		\caption{Three quotient lattices of $FD(\{a,b,c\})$}
	\end{figure}
	
\end{proof}

\subsection{Distributive Sublattices Part 2}

We analyse distributive sublattices of free lattices which have a width of two. Specifically, Proposition \ref{ladder} from F. Galvin and B. J\'onsson is proven in a new way. The new proof has some similarities to, but is more general than, W. Poguntke and I. Rival's arguments in \cite{OLE}.\\

In this paper, we will frequently consider the lattice $\textbf{2} \times \mathbb{Z}$ where $\mathbb{Z}$ is the partially ordered set: $\dots < -2 < -1 < 0 < 1 < 2 < \dots$. Since $\textbf{2}$ is the partially ordered set $0 < 1$, we write $\textbf{2} \times \mathbb{Z} = \{(j,k) : j = 0,1, k \in \mathbb{Z} \}$.\\

For convenience, I will define the term \emph{gadget}. Let $L$ be a lattice and let $p,q,r \in L$ be such that the following two properties hold: (1) $q < r$, $p \parallel q$, and $p \parallel r$; (2) $p \wedge q = p \wedge r$ or $p \vee q = p \vee r$. Then we denote the sublattice of $L$ generated by $\{p,q,r\}$ by $G_L(p;q,r)$ and call such a sublattice a \emph{gadget}. It turns out that a gadget can be isomorphic to four possible lattices, one of which being $\textbf{2} \times \textbf{3}$. To see this, we introduce the following concept.

\begin{definition}(\cite{FL}) \label{freely generated}
	Let $\langle P ; \leq \rangle$ be a partially ordered set. Then the lattice freely generated by $\langle P ; \leq \rangle$, $FL(\langle P ; \leq \rangle)$, is the lattice unique up to lattice isomorphism such that for any lattice $L$ and order preserving map $f : X \to L$, there is a lattice homomorphism $h : FL(\langle P ; \leq \rangle) \to L$ such that, for all $x \in X$, $h(x) = f(x)$.
\end{definition}

Given two finite chains $\textbf{m}$ and $\textbf{n}$, we use notation from \cite{FL} and write $\textbf{m} + \textbf{n}$ to denote the disjoint union of $\textbf{m}$ and $\textbf{n}$ whose partial order is the union of the partial order on $\textbf{m}$ and the partial order on $\textbf{n}$. \\

Consider the lattice $FL(\textbf{1} + \textbf{2})$, depicted in Figure 3, where the elements of $\textbf{1} + \textbf{2}$ are denoted $\{a,b,c\}$ and satisfy: $c \leq c$ and $b < c$. There is a unique lattice homomorphism $\phi: FL(P) \twoheadrightarrow G_L(p;q,r)$ such that $\phi(a) = p$, $\phi(b) = q$, and $\phi(c) = r$. Hence, $G_L(p;q,r)$ is isomorphic to a quotient lattice of $FL(\langle P ; \leq \rangle)$. The gadget $G_L(p;q,r)$ can be identified with the three lattices depicted in Figure 4 and their duals. This can be confirmed using the fact that all congruences on $FL(\langle P ; \leq \rangle)$ are compatible with meet and join.\\

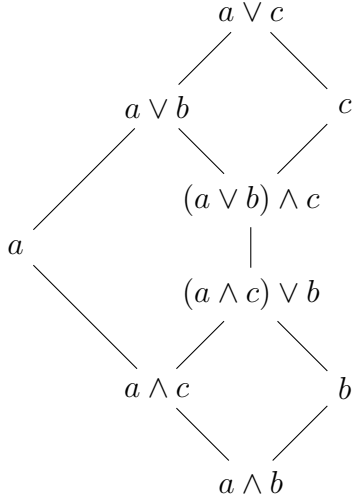
\begin{figure} \label{gadget}
	\begin{center}
		\begin{tikzpicture}[scale = 1.25]
		\node (A+C) at (0,3) {$a \vee c$};
		\node (A+B) at (-1,2) {$a \vee b$};
		\node (C) at (1,2) {$c$};
		\node (A+B C) at (0,1) {$(a \vee b) \wedge c$};
		\node (A) at (-2.5, 0.5) {$a$};
		\node (AC+B) at (0,0) {$(a \wedge c) \vee b$};
		\node (AC) at (-1,-1) {$a \wedge c$};
		\node (B) at (1,-1) {$b$};
		\node (AB) at (0,-2) {$a \wedge b$};
		
		\draw (A+B) -- (A+B C);
		\draw (A+C) -- (A+B) -- (A) -- (AC) -- (AC+B) -- (A+B C) -- (C) -- (A+C);
		\draw (AC) -- (AB) -- (B) -- (AC+B);
		\end{tikzpicture}
	\end{center}
	
	\caption{The lattice $FL(\textbf{1} + \textbf{2})$.}
	
\end{figure}
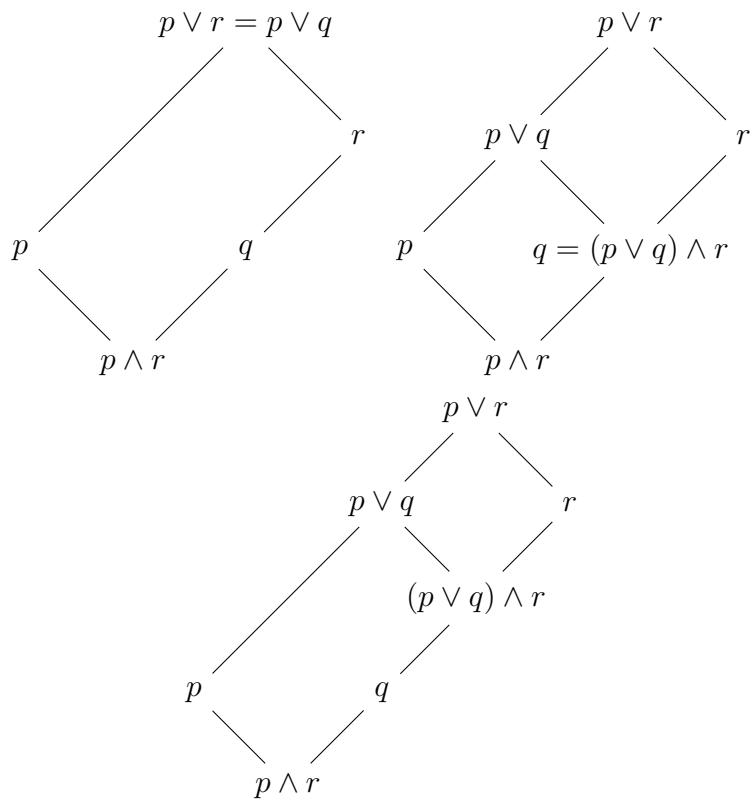
\begin{figure} \label{ThreeCases}
	
	\begin{center}
		\begin{tikzpicture}[scale = 1.5]
		
		\node (a+b) at (0,0) {$p \vee r = p \vee q$};
		\node (c) at (1,-1) {$r$};
		\node (b) at (0,-2) {$q$};
		\node (a) at (-2,-2) {$p$};
		\node (ac) at (-1,-3) {$p \wedge r$};
		
		\draw (a+b) -- (a);
		\draw (ac) -- (b) -- (c);
		\draw (a) -- (ac);
		\draw (a+b) -- (c);
		
		\end{tikzpicture}
		\begin{tikzpicture}[scale=1.5]
		
		\node (a+c) at (0,0) {$p \vee r$};
		\node (a+b) at (-1,-1) {$p \vee q$};
		\node (c) at (1,-1) {$r$};
		\node (b) at (0,-2) {$q = (p \vee q) \wedge r$};
		\node (a) at (-2,-2) {$p$};
		\node (ac) at (-1,-3) {$p \wedge r$};
		
		\draw (a+c) -- (a+b) -- (a);
		\draw (ac) -- (b) -- (c);
		\draw (a) -- (ac);
		\draw (a+b) -- (b);
		\draw (a+c) -- (c);
		
		\end{tikzpicture}
		\begin{tikzpicture}[scale = 1.25]
		
		\node (a+c) at (0,3) {$p \vee r$};
		\node (a+b) at (-1,2) {$p \vee q$};
		\node (c) at (1,2) {$r$};
		\node (a+b c) at (0,1) {$(p \vee q) \wedge r$};
		\node (a) at (-3, 0) {$p$};
		\node (b) at (-1,0) {$q$};
		\node (ac) at (-2,-1) {$p \wedge r$};
		
		\draw (a+b) -- (a+b c);
		\draw (a+c) -- (a+b) -- (a) -- (ac) -- (b) -- (a+b c) -- (c) -- (a+c);
		\end{tikzpicture}
	\end{center}
	
	\caption{Three cases.}
	
\end{figure}

The new proof for Proposition \ref{ladder} uses a well-known theorem in lattice theory known as the $M_3-N_5$ theorem (see \cite{ILO}). Recall the lattices $M_3$, the diamond, and $N_5$, the pentagon (see \cite{ILO}). The $M_3-N_5$ theorem implies that a distributive lattices cannot have a sublattice isomorphic to $N_5$; we will use this property.\\

My proof of Proposition \ref{ladder} is as follows. The last part of the proof can be simplified if only countable lattices were considered; the Axiom of Choice is used for uncountable lattices.\\

\begin{proof} (Brian T. Chan) Let $D$ be a linearly indecomposable distributive lattice of width two that has no doubly reducible elements. It is easy to see that if $|D| \leq 4$, then $D \cong \textbf{2} \times \textbf{2}$. So assume that $|D| \geq 5$. Then $D$ contains a \emph{proper} sublattice isomorphic to $\textbf{2} \times \textbf{2}$. Because $D$ is linearly indecomposable and has no doubly reducible elements it is not hard to see the following: (1) $|D| > 4$ implies that $D$ has at least one gadget; (2) For all \emph{proper} sublattices $H$ of $D$ isomorphic to $\textbf{2} \times C$ for some chain $C$, there is an element $z_H \in D \backslash H$ such that $z_H$ is not an upper bound or a lower bound of $H$.\\

In ZFC set theory, the Axiom of Choice is equivalent to the Well-Ordering Theorem. The Well-Ordering Theorem implies that there is a least ordinal $\alpha$ such that the set of elements of $D$ can be written as $\{p_j : j \in \alpha \}$. In particular, $\alpha = \omega$ if $D$ is countably infinite. \\

Let $G_D(p;q,r)$ be a gadget of $D$ for some $p,q,r \in D$. Because $D$ is modular, the $M_3 - N_5$ theorem indicates that $G_D(p;q,r)$ is the lattice in Figure 3.4 isomorphic to $\textbf{2} \times \textbf{3}$ or the dual of that lattice. So we identify $G_D(p;q,r)$ with $\langle \{s,t,u,x,y,z\} ; \leq \rangle$ where the partial order $\leq$ is as depicted in Figure 3.5. \\

Let $J \subseteq \alpha$ be the \emph{non empty} set of indices $j$ such that $p_j \notin G_D(p;q,r)$ and $p_j$ is neither an upper bound nor a lower bound of $G_D(p;q,r)$. Setting $i$ to be the least element of $J$ in $\alpha$, let $w = p_i \in L \backslash G_D(p;q,r)$. Using Figure 3.5 as a reference, the sublattice generated by $\{p,q,r,w\}$ will be shown to be isomorphic to $\textbf{2} \times \textbf{4}$. Since the width of $D$ is two, it is enough to consider the following cases: \\

\begin{figure} \label{three gadgets}

\begin{center}
\begin{tikzpicture}[scale = 1.25]

\node (u) at (0,0) {$u$};
\node (t) at (-1,-1) {$t$};
\node (z) at (1,-1) {$z$};
\node (y) at (0,-2) {$y$};
\node (s) at (-2,-2) {$s$};
\node (x) at (-1,-3) {$x$};

\draw (u) -- (t) -- (s);
\draw (x) -- (y) -- (z);
\draw (s) -- (x);
\draw (t) -- (y);
\draw (u) -- (z);

\end{tikzpicture}
\end{center}

\caption{A gadget in a distributive lattice.}

\end{figure}
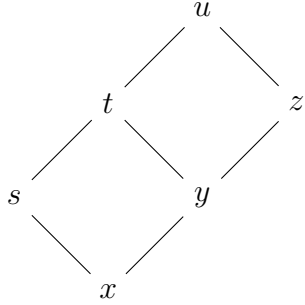

Case 1 : If $u \wedge w = z$ or $x \vee w = s$, then the sublattice generated by $\{p,q,r,w\}$ can be seen to be isomorphic to $\textbf{2} \times \textbf{4}$.\\

Case 2 : Assume that $s < w < t$. Then $\{y,s,q\}$ generates a gadget and as $D$ is modular, $G_D(y;s,w) \cong \textbf{2} \times \textbf{3}$. So it follows that the sublattice of $D$ generated by $\{p,q,r,w\}$ is isomorphic to $\textbf{2} \times \textbf{4}$. The other three cases $t < w < u$, $x < w < y$, and $y < w < z$ follow by symmetry.\\

Case 3 : Assume that $w > w \wedge u \geq t$. Since $t$ is not doubly reducible, $w \wedge u > t$. So $G_D(z;t,w \wedge u)$ is a gadget and by the argument from Case 2, $w \wedge u$ is join reducible. But then, $w \wedge u$ is doubly reducible which is impossible. A dual argument applies to when $w < w \vee x \leq y$.\\

Case 4 : Suppose that $z < w < u$. We see that $t \parallel w$, so consider the gadget $G_D(t;z,w)$. Since $D$ is modular, $G_D(t;z,w) \cong 2 \times 3$. Moreover, because $t \vee w = t \vee z = u$, it follows that $y < t \wedge w < t$. But then $\{s, t \wedge w, z\}$ is an antichain of $D$ contrary to assumption. A dual argument shows that $x < w < s$ is also impossible. \\

Case 5 : Supposing that $y < w < t$ implies that $\{s,w,z\}$ is an antichain. But this is impossible as the width of $D$ is two.\\

Hence, the sublattice $D_1$ generated by $\{p,q,r,w\}$ is isomorphic to $\textbf{2} \times \textbf{4}$. We now carry on a proof by induction. If $D_1 \neq L$, then let $i$ be the least element of $\alpha$ such that $p_i \notin D_1$ and $p_i$ is neither an upper bound nor a lower bound of $D_1$. Set $w_2 = p_i \in L \backslash D_1$. Because the width of $D$ is two, the above argument implies that the sublattice $D_2$ of $D$ generated by $\{p,q,r,w,w_2\}$ is isomorphic to $\textbf{2} \times \textbf{5}$. \\

When $D$ is countable, repeating the above argument indefinitely proves the proposition by induction since $\alpha = \omega$ or $\alpha$ is finite. When $D$ is uncountable, we use transfinite induction: \\

We have nested sequence of lattices $D_1 \subseteq D_2 \subseteq \dots$ isomorphic to $\textbf{2} \times \textbf{3}, \textbf{2} \times \textbf{4}, \dots$ Now assume that for all ordinals $\zeta$ where $D_\zeta$ is defined, $D_\zeta \subseteq D$ and $D_\zeta \cong \textbf{2} \times C_\zeta$ for some chain $C_\zeta$ dependent on $\zeta$ where $|C_\zeta| = |\zeta|$; and assume that for all ordinals $\zeta_1 \leq \zeta_2$ where $D_{\zeta_1}$ and $D_{\beta_2}$ are defined, $D_{\zeta_1} \subseteq D_{\zeta_2}$. For all limit ordinals $\kappa$ such that $D_\zeta$ is defined for all ordinals $\zeta < \kappa$, set $D_\gamma = \cup_{\zeta \in \gamma} D_\zeta \subseteq D$. Now assume that $\kappa$ is the least successor ordinal such that $D_\kappa$ is undefined and $D \neq D_{\kappa'}$ where $\kappa'$ is the ordinal such that $\kappa = \kappa' + 1$. \\

Since $D_{\kappa'} \neq D$, let $i$ be the least ordinal in $\alpha$ such that $p_i \notin D_{\kappa'}$ and $p_i$ is not an upper bound or a lower bound of $D_{\kappa'}$, and set $w_\kappa = p_i$. Assume without loss of generality that $D_{\kappa'} = \textbf{2} \times C_{\kappa'}$ for some chain $C_{\kappa'}$. Set $I = \{q \in C_{\kappa'} : (0,q) \leq w_\kappa \}$ and $F = \{q \in C_{\kappa'} : (1,q) \geq w_\kappa \}$. Since the width of $D$ is two, it follows that $D_{\kappa'} = (\textbf{2} \times I) \cup (\textbf{2} \times F)$, $(\textbf{2} \times I) \cap (\textbf{2} \times F) = \varnothing$, $\textbf{2} \times I$ is empty or an ideal of $D_{\kappa'}$, and $\textbf{2} \times F$ is empty or a filter of $D_{\kappa'}$. \\

So define $D_\kappa$ to be the sublattice of $D$ generated by $D_{\kappa'} \cup \{w_\kappa \}$. By the above arguments used to construct $D_k$ when $k$ is finite, it follows that $D_\kappa \cong \textbf{2} \times C_\kappa$ for some chain $C_\kappa$ satisfying $|C_{\kappa}| = \kappa$. So as the elements of $D$ are well ordered by $\alpha$, we see that for some ordinal $\lambda$ satisfying $|\lambda| \leq \alpha$, $D = D_\lambda$. This completes the proof.
\end{proof}

\section{Semidistributive Varieties}

We consider how developments in Section 1.1 could possibly be generalized. However, we warn the reader that pursuing the strategies proposed may be a very arduous task. \\

The new idea used in Section 1.1 was to consider the free distributive lattice on three generators, $FD(3) = F_\mathcal{D}(3)$ where $\mathcal{D}$ denotes the variety of distributive lattices. Such a structure can be generalized, and if enough is known about such a generalization perhaps more countable sublattices of free lattices can be identified. We now explore some of these possibilities below.\\

A variety $\mathcal{V}$ of lattices is defined to be \emph{semidistributive} (see \cite{VL}) if every lattice in $\mathcal{V}$ is semidistributive. If $\mathcal{S}$ denotes the class of all semidistributive lattices, we have that $\mathcal{D} \subseteq \mathcal{V} \subseteq \mathcal{S}$. So a possible idea is to consider the relatively free lattices $F_\mathcal{V}(3)$ when $\mathcal{V}$ is a semidistributive variety. We note that there are semidistributive lattices which do not belong in any semidistributive variety, for instance the free lattice $FL(3)$. \\

The number of semidistributive varieties turns out to be countable. In 1979 B. J\'onsson and I. Rival had characterized semidistributive varieties in \cite{SNV}, and part of their characterization is as follows (see \cite{VL}):\\

\begin{theorem}(B. J\'onsson and I. Rival) \label{extended semidistributive laws}
	
	The following are equivalent for a variety $\mathcal{V}$ of lattices:
	
	(1) $\mathcal{V}$ is semidistributive.
	
	(2) Let $y_0 = y$, $z_0 = z$, and for all non-negative integers $k$ let $y_{k+1} = y \vee (x \wedge z_k)$ and let $z_{k+1} = z \vee (x \wedge y_k)$. Then for some nonnegative integer $n$, the identity
	
	\begin{center}
		$(SD_n^\wedge): x \wedge (y \vee z) = x \wedge y_n$
	\end{center}
	
	and its dual $(SD_n^\vee)$ are satisfied by every lattice $L$ in $\mathcal{V}$.
\end{theorem}

Theorem \ref{extended semidistributive laws} is proven in \cite{SNV} and \cite{VL}. We write $\mathcal{S}_n$ to denote the variety of lattices which satisfy $(SD_n^\wedge)$ and $(SD_n^\vee)$. We have this: \\

\begin{center}
	$\mathcal{D} = \mathcal{S}_1 \subset \mathcal{S}_2 \subset \mathcal{S}_3 \subset \dots \subset S$
\end{center}

This motivates us to study the following sequence of lattices: \\
$F_{\mathcal{S}_1}(3), F_{\mathcal{S}_2}(3), F_{\mathcal{S}_3}(3), F_{\mathcal{S}_4}(3), \dots$. We note that we have this sequence of onto homomorphisms:

\begin{center}
	$FD(3) = F_{\mathcal{S}_1}(3) \leftarrow F_{\mathcal{S}_2}(3) \leftarrow F_{\mathcal{S}_3}(3) \leftarrow F_{\mathcal{S}_4}(3) \leftarrow \dots$
\end{center}

The new proof in Section 1.1, the result of Section 1.1 being largely responsible for characterizing distributive sublattices of free lattices up to isomorphism, relied on knowledge of the lattice $F_{\mathcal{S}_1}(3) = FD(3)$. Perhaps if we knew more about the relatively free lattices $F_{\mathcal{S}_2}(3), F_{\mathcal{S}_3}(3), F_{\mathcal{S}_4}(3), F_{\mathcal{S}_5}(3), \dots$ we could make progress on the following problem:

\begin{problem} \label{odessey}
	Which countable lattices belonging to a semidistributive variety are isomorphic to a sublattice of a free lattice?
\end{problem}

F. Galvin and B. J\'onsson's paper \cite{DSFL} answered Problem \ref{odessey} for the variety of distributive lattices. But not much is known even for the variety $\mathcal{S}_2$ of \emph{nearsemidistributive} lattices. \\

However, if the lattice $F_{\mathcal{S}_2}(3)$ is finite, it may be possible to extend the new proof from Section 1.1 to analyse near semidistributive sublattices of free lattices. I will make the following observations:\\

\begin{proposition} \label{partial answer}
	
Let $\mathcal{S}_n^\vee$ denote the variety of lattices which satisfy $(SD_n^\vee)$ and let  $\mathcal{S}_n^\wedge$ denote the variety of lattices which satisfy $(SD_n^\wedge)$. Then $F_{\mathcal{S}_2^\vee}(3)$ and $F_{\mathcal{S}_2^\wedge}(3)$ are countably infinite lattices.

\end{proposition}

\begin{proof} (Brian T. Chan) Consider the lattice $Z_d$ and the elements $a$, $b$, and $c$ as depicted p 214 of in J. Reinhold's paper \cite{WDLFL}; and let $Z_d'$ denote the infinite sublattice of $Z_d$ generated by $a,b,c$. Since the lattice $Z_d'$ satisfies $(SD^\vee_2)$, $Z_d' \in \mathcal{S}_2^\vee$. So there exists an onto homomorphism $f: F_{\mathcal{S}_2^\vee}(3) \to Z_d'$ induced by mapping the three generators of $F_{\mathcal{S}_2^\vee}(3)$ to $\{a,b,c\} \subseteq Z_d'$, implying that $F_{\mathcal{S}_2^\vee}(3)$ is infinite (and hence $F_{\mathcal{S}_n^\vee}(3)$ is infinite for $n \geq 3$) since $Z_d'$ is an infinite lattice. A dual argument can be used for $F_{\mathcal{S}_2^\wedge}(n)$ when $n \geq 3$\\

\end{proof}

More generally, we can consider the lattices $F_{\mathcal{S}_k}(n)$ for positive integers $k$ and $n$. 
If $F_{\mathcal{S}_k}(n)$ is a finite bounded lattice, then by A. Day's characterization theorem, $F_{\mathcal{S}_k}(n)$ can be obtained from a one element lattice by a sequence of interval doublings; this construction will be explained in more detail in the next section. This could be a very useful for generalizing the new proof from Section 1.1 since we may be able to construct $F_{\mathcal{S}_k}(n)$ more easily. I will make the following observation. Recall that the notion of a bounded lattice does not require the existence of maximum or minimum elements (see \cite{FL} and \cite{VL} for definitions).

\begin{proposition} \label{partial answer 2}
Let $n$ be a positive integer. If $F_{\mathcal{S}_2}(n)$ is a finite, then $F_{\mathcal{S}_2}(n)$ is a bounded lattice. 
\end{proposition}
\begin{proof}(Brian T. Chan) The lattices $L_{11}$ and $L_{12}$ as described in \cite{VL} and H.Rose's observation from \cite{VL} regarding these lattices and $C-$cycles on finite semidistributive lattices will be used. See \cite{VL} and \cite{FL} for notation or concepts being used. Assume that $F_{\mathcal{S}_2}(n)$ is finite. If $F_{\mathcal{S}_2}(n)$ is not bounded, then without loss of generality assume that $F_{\mathcal{S}_2}(n) \neq D(F_{\mathcal{S}_2}(n))$. As $F_{\mathcal{S}_2}(n)$ is finite, there is a $C-$cycle in $F_{\mathcal{S}_2}(n)$. By H. Rose's observation, it follows that $F_{\mathcal{S}_2}(n)$ contains a sublattice isomorphic to $L_{11}$ or $L_{12}$. But the lattices $L_{11}$ and $L_{12}$ are not neardistributive, which is an impossibility.

\end{proof}

We now see how the new proofs for F. Galvin and B. J\'onsson's work can possibly be extended in another way.

\section{Spanning Pairs and Finite Width Sublattices of Free Lattices}

The width of a lattice is the supremum of the cardinalities of all antichains in that lattice. In particular, a finite width lattice is a lattice such that for some positive integer $k$, every antichain has cardinality at most $k$. The machinery used in Section 1.1 reduced the number of possible cases to consider when analysing distributive sublattices. However, the machinery used in Section 1.2 is related to how distributive sublattices of free lattices can be constructed. We consider the following idea: A subset $C \subseteq L$ of a lattice $L$ is \emph{convex} if for all $p,q,r \in L$, $p,q \in C$ and $p \leq r \leq q$ implies $r \in C$. We introduce A. Day's \emph{doubling construction} following \cite{FL}:\\

\begin{definition} [A. Day's Doubling Construction]
	Let $L$ be a lattice and $C$ be a convex subset of $L$. Define $L[C]$ to be the disjoint union $(L \backslash C) \cup (C \times 2)$ with $p \leq q$ if one of the following holds.\\
	
	(1) $p,q \in L \backslash C$ and $p \leq q$ holds in $L$\\
	
	(2) $p,q \in C \times 2$ and $p \leq q$ holds in $C \times 2$\\
	
	(3) $p \in L \backslash C$, $q = (u,i) \in I \times 2$, and $p \leq u$ holds in $L$\\
	
	(4) $p = (v, i) \in C \times 2$, $q \in L \backslash C$, and $v \leq q$ holds in $L$.
\end{definition}

Such a construction is an interval doubling construction if $C$ is assumed to be an interval: with $p \leq q$ in $L$, $C = \{r \in L : p \leq r \leq q \}$. As evident in Theorem \ref{the characterization}, all distributive sublattices of free lattices can be constructed from a \emph{countable chain} by repeatedly applying A. Day's doubling construction. This idea also applies to finite sublattices of free lattices: A result from A. Day (see \cite{FBL}, \cite{VL}, and \cite{FL}) implies that every finite sublattice of a free lattice can be constructed from a \emph{one element lattice} by repeatedly applying A. Day's interval doubling construction. I will prove a result which restricts how A.Day's doubling construction can be used for sublattices of free lattices by introducing the notion of a \emph{spanning pair}. We write $p \prec q$ to mean that $q$ \emph{covers} $p$ (this means that if $p \leq r \leq q$ then $r = p$ or $r = q$.)

\begin{definition}
Let $L$ be a lattice with no minimal element or maximal element. A spanning pair $|p,q|$ consists of two elements $p \prec q$ in $L$ and two sequences $p < p_1 < p_2 < \dots$ and $q > q_1 > q_2 > \dots$ in $L$ such that: $q,q_1,q_2,\dots$ has no lower bound in $L$, $p,p_1,p_2,\dots$ has no upper bound in $L$, $q \nleq p_m$ for all $m$, and $p \ngeq q_n$ for all $n$. We say that a spanning pair $|p,q|$ is induced by $(p,q)$.
\end{definition}

If $L$ is a lattice and $C$ is a convex subset of $L$, then it is not hard to see that a necessary condition for the lattice $L[C]$ to satisfy Whitman's condition is that every non maximal element and every non minimal element of $C$ must be doubly irreducible. In particular, the following can be said as lattice congruences are compatible with join and meet. If $|p,q|$ is a spanning pair of $L[C]$, $L[C]$ satisfies Whitman's condition, and $p = (0,c)$ and $q = (1,c)$ for some $c \in C$, then $C$ is a chain because $\textbf{2} \times C$ contains $|p,q|$.\\

This observation can be carried further. We prove Theorem \ref{spanning pair} which is a new result concerning lattices with no doubly reducible elements. In summary, Theorem \ref{spanning pair} asserts the following in almost all cases: If $L$ is a sublattice of a free lattice and if there exists another lattice $L'$ such that $L \cong L'[C]$ for some convex subset $C$ of $L$, then $C$ must ``avoid'' all spanning pairs of $L$. Moreover, the proof of Theorem \ref{spanning pair} uses \emph{gadgets}.\\

\begin{theorem} \label{spanning pair}
Let $L$ be a countable lattice with no doubly reducible elements. If there is a spanning pair $|p,q|$ of $L$, a sublattice $L' \subseteq L$, a convex subset $C$ of $L'$, and an isomorphism $f: L \to L'[C]$ such that $f(p) = (0,c)$ and $f(q) = (1,c)$ for some $c \in C$, then $L \cong \textbf{2} \times Z$ for some chain $Z$.
\end{theorem}

\begin{proof} (Brian T. Chan) Let $|p,q|$ be a spanning pair of $L$ with $p < p_1 < p_2 < \dots$ and $q > q_1 > q_2 > \dots$ Because $|p,q|$ is a spanning pair, we can define the following infinite subsequence $p_{n_1} < p_{n_2} < \dots$ of $p < p_1 < p_2 < \dots$:\\

Let $p_{n_1} = p_1$, and for all positive integers $k$, let $p_{n_{k+1}}$ satisfy $q \vee p_{n_k} < q \vee p_{n_{k+1}}$. Dually, define a subsequence $q > q_{n_1} > q_{n_2} > \dots$ of $q > q_1 > q_2 > \dots$ in the same way that $p < p_{n_1} < p_{n_2} < \dots$ was defined. Now we show \textbf{(1)}:\\

\begin{center}
\textbf{(1)} \emph{There exists a positive integer $k$ and an element $p_1'$ such that $p_{n_k} \leq p_1' \prec q \vee p_{n_k}$.}
\end{center}

For all positive integers $k$, set $a_k = q \vee p_{n_k}$ and $p_{0,k} = p_{n_k}$. Suppose that \textbf{(1)} is false, we first show by contradiction the following.\\

\begin{center}
\textbf{(2)} \emph{There is a positive integer $m_1$ and an element $r_1$ such that $p_{0,m_1} < r_1 < a_{m_1}$ and the following occurs with $p_{1,k} = p_{0,m_1 + k - 1}$: For all positive integers $k$, $r_1 \vee p_{1,k} > p_{1,k}$.}
\end{center}

Since \textbf{(1)} is assumed to be false, let $r_{1,1}$ satisfy $a_1 \wedge p_{0,2} < r_{1,1} < a_1$. Then as $r_{1,1} \nleq p_{0,2}$, $r_{1,1} \vee p_{0,2} > p_{0,2}$. If $r_{1,1} \vee p_{0,k} > p_{0,k}$ for all $k$, then set $m_1 = 1$ and $r_1 = r_{1,1}$. Otherwise, let $m' > 2$ be the least positive integer such that $r_{1,1} \leq p_{0,m'}$. Now pick an $r_{1,2}$ such that $a_{m'} \wedge p_{0,m'+1} < r_{1,2} < a_{m'}$, $r_{1,2} \wedge a_2 > r_{1,1} \vee p_{0,2}$, and $r_{1,2} \wedge a_1 > r_{1,1}$. Such an element $r_{1,2}$ exists because \textbf{(1)} is assumed to be false and $p \prec q$. Otherwise, there is an $s$ such that $a_{m'} \wedge p_{0,m'+1} < s < a_{m'}$ and a $t$ such that $r_1 < t < a_1$ or $r_1 \vee p_{0,2} < t < a_2$ which satisfy the following: $q \wedge s = q \wedge t = p$, but $q \wedge (s \vee t) = q \wedge a_{m'} = q$ which would violate the meet semidistributive laws. \\

Like before, $r_{1,2} \vee (r_1 \vee p_{0,m'+1}) > r_1 \vee p_{0,m'+1}$ and for some positive integer $m''$ satisfying $m'' \geq m' + 2$, $r_{1,2} \leq r_1 \vee p_{0,m''}$. By the above arguments, there is an element $r_{1,3}$ such that: $a_{m''} \wedge p_{0,m''+1} < r_{1,3} < a_{m''}$, $r_{1,3} \wedge a_{m'+1} > r_{1,2} \vee p_{0,m'+1}$, $r_{1,3} \wedge a_{m'} > r_{1,2}$, $r_{1,3} \wedge a_2 > r_{1,2} \wedge a_2$, and $r_{1,3} \wedge a_1 > r_{1,2} \wedge a_1$. As we are supposing that both \textbf{(1)} and \textbf{(2)} are false, this process can be continued indefinitely. But then there is a sequence of antichains $B_2, B_3, B_4, \dots$ in $L$ where $B_2 = \{ r_{1,2} \wedge a_2, r_{1,3} \wedge a_1 \}, B_3 = \{ r_{1,3} \wedge a_{m'}, r_{1,4} \wedge a_2, r_{1,5} \wedge a_1 \}, \dots$, $|B_k| = k$ for all $k$, and every element of $B_k$ is meet reducible. But this is contrary to assumption. Hence, \textbf{(2)} follows assuming that \textbf{(1)} is false.\\

We note that $r_1 < r_1 \vee p_{1,2} < r_1 \vee p_{1,3} < \dots$. Since $p \prec q$, $q \wedge r_1 = q \wedge p_{1,k} = p$ for all $k \geq 2$. So by the meet semidistributive laws, $p_{1,k} < r_1 \vee p_{1,k} < q \vee p_{1,k} = a_{m_1 + k - 1}$ for all $k \geq 2$. Since \textbf{(1)} is assumed to be false, the above proof of \textbf{(2)}, can be generalized to show that there is a positive integer $m_2 \geq m_1$ and an element $r_2$ such that $r_1 \vee p_{1,m_2} < r_2 < a_{m_2}$ and the following occurs with $p_{2,k} = r_1 \vee p_{1,m_2 + k - 1}$: For all positive integers $k$, $r_1 \vee p_{2,k} > p_{2,k}$. We note that $(p,q)$ with the sequences $p < r_2 \vee p_{2,1} < r_2 \vee p_{2,2} < r_2 \vee p_{2,3} < \dots$ and $q > q_1 > q_2 > \dots$ also form a spanning pair of $L$. Hence, we can replicate the above arguments to produce the following infinite sequence of sequences $(p_{n,k})_k$ and infinite sequence of elements $r_n$ satisfying $p_{n,1} < r_n < a_{m_n}$ for all $n = 0,1,2,3,\dots$. In particular, there is a sequence of antichains $A_2, A_3, A_4, \dots$ in $L$ where $A_2 = \{ r_2 \vee p_{2,2}, p_{2,3} \}, A_3 = \{r_3 \vee p_{3,2}, p_{3,3}, r_1 \vee p_{0, m_3 + k - 1}\}, \dots$, every element of $A_k$ is join reducible, and for all $k$, $|A_k| = k$. But this is also contrary to assumption.\\

Hence, \textbf{(1)} is true. Let $p_1'$ and $k$ be as described in \textbf{(1)}. To see that $p_1' \vee p_{n_j} < q \vee p_{n_j}$ for all $j \geq k$, suppose that $p_1' \vee p_{n_j} = q \vee p_{n_j}$. Since $p \prec q$, $q \wedge p_1' = q \wedge p_{n_j} = p$ and $q \wedge (p_1' \vee p_{n_j}) = q \wedge (q \vee p_{n_j}) = q$. But that is impossible as $L$ is meet semidistributive. Using assertion \textbf{(1)} indefinitely, we obtain a chain $p < p_1' < p_2' < \dots$ such: $(p,q)$, $p < p_1' < p_2' < \dots$, and $q > q_1 > q_2 > \dots$ form a spanning pair of $L$; and for all positive integers $k$, $p_k' \prec q \vee p_k'$. \\

Dually, we use the join semidistributive laws and the assumption that $L$ has finite width to construct a chain $q > q_1' > q_2' > \dots$ such that: $(p,q)$, $p < p_1' < p_2' < \dots$, and $q > q_1' > q_2' > \dots$ form a spanning pair of $L$; and for all positive integers $k$, $p \wedge q_k' \prec q_k'$. \\

Hence, the sublattice of $L$ generated by $\{p, p_1', p_2', \dots \} \cup \{q, q_1', q_2', \dots \}$ is isomorphic to $\textbf{2} \times \mathbb{Z}$. We denote this sublattice by $\textbf{2} \times \mathbb{Z}$ and denote its elements by $(i,k)$.\\

To prove the last part of the theorem, let $r \in L \backslash (\textbf{2} \times \mathbb{Z})$. Suppose that $(0,m) < r < (1,n)$ for some integers $m \leq n$ where $(1,m) \nleq r$ and $r \nleq (0,n)$. Since $(0,n) \prec (1,n)$ in $L$, $(1,m-1) \wedge r = (1,m-1) \wedge (0,n) < (1,m-1) \wedge (r \vee (0,n))$. But this contradicts the fact that $L$ join semidistributive. This completes the proof.\\
\end{proof}

Hence, for lattices not isomorphic to $\textbf{2} \times Z$ for some chain $Z$, spanning pairs give a way of identifying where A. Day's doubling construction cannot be used. It would be nice if there were a way to more easily recognize where such spanning pairs may appear. For this, I will create and prove Theorem \ref{spanning pair SD}.

\begin{theorem} \label{spanning pair SD}

Let $L$ be a semidistributive lattice and assume that there exists a positive integer $N$ satisfying the following: If $A$ is an antichain of $L$ where every element of $A$ is join reducible or meet reducible, then $|A| \leq N$.\\

If $|p,q|$ is a spanning pair, then there exists a lattice embedding $f : \textbf{2} \times \mathbb{Z} \to L$ such that $p = f(0,0)$, $q = f(0,1)$, and $f(0,k) \prec f(1,k)$ in $L$ for all $k \in \mathbb{Z}$. Moreover, if $r \in L \backslash f(\textbf{2} \times \mathbb{Z})$, then $(0,m) < r < (1,n)$ for some integers $m \leq n$ implies that $r \leq (0,n)$ or $(1,m) \leq r$.

\end{theorem}

\begin{proof} (Brian T. Chan) Let $|p,q|$ be a spanning pair of $L$ with $p < p_1 < p_2 < \dots$ and $q > q_1 > q_2 > \dots$ Because $|p,q|$ is a spanning pair, we can define the following infinite subsequence $p_{n_1} < p_{n_2} < \dots$ of $p < p_1 < p_2 < \dots$:\\

Let $p_{n_1} = p_1$, and for all positive integers $k$, let $p_{n_{k+1}}$ satisfy $q \vee p_{n_k} < q \vee p_{n_{k+1}}$. Dually, define a subsequence $q > q_{n_1} > q_{n_2} > \dots$ of $q > q_1 > q_2 > \dots$ in the same way that $p < p_{n_1} < p_{n_2} < \dots$ was defined. Now we show \textbf{(1)}:

\begin{center}
\textbf{(1)} \emph{There exists a positive integer $k$ and an element $p_1'$ such that $p_{n_k} \leq p_1' \prec q \vee p_{n_k}$.}
\end{center}

For all positive integers $k$, set $a_k = q \vee p_{n_k}$ and $p_{0,k} = p_{n_k}$. Suppose that \textbf{(1)} is false, we first show by contradiction the following.

\begin{center}
\textbf{(2)} \emph{There is a positive integer $m_1$ and an element $r_1$ such that $p_{0,m_1} < r_1 < a_{m_1}$ and the following occurs with $p_{1,k} = p_{0,m_1 + k - 1}$: For all positive integers $k$, $r_1 \vee p_{1,k} > p_{1,k}$.}
\end{center}

Since \textbf{(1)} is assumed to be false, let $r_{1,1}$ satisfy $p_{0,1} \leq a_1 \wedge p_{0,2} < r_{1,1} < a_1$. Then as $r_{1,1} \nleq p_{0,2}$, $r_{1,1} \vee p_{0,2} > p_{0,2}$. If $r_{1,1} \vee p_{0,k} > p_{0,k}$ for all $k$, then set $m_1 = 1$ and $r_1 = r_{1,1}$. Otherwise, let $m' > 2$ be the least positive integer such that $r_{1,1} \leq p_{0,m'}$. We note that for all $2 \leq k < m'$, $r_{1,1} \vee p_{0,k} < a_k$. Now pick an $r_{1,2}$ such that $p_{0,m'} \leq a_{m'} \wedge p_{0,m'+1} < r_{1,2} < a_{m'}$, $r_{1,2} \wedge a_2 > r_{1,1} \vee p_{0,2}$, and $r_{1,2} \wedge a_1 > r_{1,1}$. Such an element $r_{1,2}$ exists because \textbf{(1)} is assumed to be false and $p \prec q$. Otherwise, there is an $s$ such that $a_{m'} \wedge p_{0,m'+1} < s < a_{m'}$ and a $t$ such that $r_{1,1} < t < a_1$ or $r_{1,1} \vee p_{0,2} < t < a_2$ which satisfy the following: $q \wedge s = q \wedge t = p$, but $q \wedge (s \vee t) = q \wedge a_{m'} = q$ which would violate the meet semidistributive laws.\\

Like before, $r_{1,2} \vee p_{0,m'+1} > p_{0,m'+1}$ and for some positive integer $m''$ satisfying $m'' \geq m' + 2$, $r_{1,2} \leq p_{0,m''}$. We note that for all $m' < k < m''$, $r_{1,2} \vee p_{0,k} < a_k$. By the above arguments, there is an element $r_{1,3}$ such that: $p_{0,m''} \leq a_{m''} \wedge p_{0,m''+1} < r_{1,3} < a_{m''}$, $r_{1,3} \wedge a_{m'+1} > r_{1,2} \vee p_{0,m'+1}$, $r_{1,3} \wedge a_{m'} > r_{1,2}$, $r_{1,3} \wedge a_2 > r_{1,2} \wedge a_2$, and $r_{1,3} \wedge a_1 > r_{1,2} \wedge a_1$. As we are supposing that both \textbf{(1)} and \textbf{(2)} are false, this process can be continued indefinitely. But then there is a sequence of antichains $B_2, B_3, B_4, \dots$ in $L$ where $B_2 = \{ r_{1,2} \wedge a_2, r_{1,3} \wedge a_1 \}, B_3 = \{ r_{1,3} \wedge a_{m'+1}, r_{1,4} \wedge a_{m'}, r_{1,5} \wedge a_2 \}, \dots$, $|B_k| = k$ for all $k$, and every element of $B_k$ is meet reducible. But this is contrary to assumption. Hence, \textbf{(2)} follows assuming that \textbf{(1)} is false.\\

We note that $r_1 < r_1 \vee p_{1,2} < r_1 \vee p_{1,3} < \dots$. Since $p \prec q$, $q \wedge r_1 = q \wedge p_{1,k} = p$ for all $k \geq 2$. So by the meet semidistributive laws, $p_{1,k} < r_1 \vee p_{1,k} < q \vee p_{1,k} = a_{m_1 + k - 1}$ for all $k \geq 2$. Since \textbf{(1)} is assumed to be false, the above proof of \textbf{(2)}, can be used to show that there is a positive integer $m_2 \geq m_1$ and an element $r_2$ such that $r_1 \vee p_{1,m_2} < r_2 < a_{m_2}$ and the following occurs with $p_{2,k} = r_1 \vee p_{1,m_2 + k - 1}$: For all positive integers $k$, $r_2 \vee p_{2,k} > p_{2,k}$. We note that $(p,q)$ with the sequences $p < r_2 \vee p_{2,1} < r_2 \vee p_{2,2} < r_2 \vee p_{2,3} < \dots$ and $q > q_1 > q_2 > \dots$ also form a spanning pair of $L$. Hence, we can replicate the above arguments to produce the following infinite sequence of sequences $(p_{0,k})_k, (p_{1,k})_k, (p_{2,k})_k, \dots$ and the following infinite sequence of elements $r_1, r_2, r_3, \dots$ satisfying $p_{n,1} < r_n < a_{m_n}$ for all $n$. In particular, there is a sequence of antichains $A_2, A_3, A_4, \dots$ in $L$ where $A_2 = \{ r_2 \vee p_{2,2}, p_{2,3} \}, A_3 = \{r_3 \vee p_{3,2}, p_{3,3}, r_1 \vee p_{0, m_3 + k - 1}\}, \dots$, every element of $A_k$ is join reducible, and for all $k$, $|A_k| = k$. But this is also contrary to assumption.\\

Hence, \textbf{(1)} is true. Let $p_1'$ and $k$ be as described in \textbf{(1)}. To see that $p_1' \vee p_{n_j} < q \vee p_{n_j}$ for all $j \geq k$, suppose that $p_1' \vee p_{n_j} = q \vee p_{n_j}$. Since $p \prec q$, $q \wedge p_1' = q \wedge p_{n_j} = p$ and $q \wedge (p_1' \vee p_{n_j}) = q \wedge (q \vee p_{n_j}) = q$. But that is impossible as $L$ is meet semidistributive. Using assertion \textbf{(1)} indefinitely, we obtain a chain $p < p_1' < p_2' < \dots$ such: $(p,q)$, $p < p_1' < p_2' < \dots$, and $q > q_1 > q_2 > \dots$ form a spanning pair of $L$; and for all positive integers $k$, $p_k' \prec q \vee p_k'$.\\

Dually, we use the join semidistributive laws and the assumption that $L$ has finite width to construct a chain $q > q_1' > q_2' > \dots$ such that: $(p,q)$, $p < p_1' < p_2' < \dots$, and $q > q_1' > q_2' > \dots$ form a spanning pair of $L$; and for all positive integers $k$, $p \wedge q_k' \prec q_k'$.\\

Hence, the sublattice of $L$ generated by $\{p, p_1', p_2', \dots \} \cup \{q, q_1', q_2', \dots \}$ is isomorphic to $\textbf{2} \times \mathbb{Z}$. We denote this sublattice by $\textbf{2} \times \mathbb{Z}$ and denote its elements by $(i,k)$.\\

To prove the last part of the theorem, let $r \in L \backslash (\textbf{2} \times \mathbb{Z})$. Suppose that $(0,m) < r < (1,n)$ for some integers $m \leq n$ where $(1,m) \nleq r$ and $r \nleq (0,n)$. Since $(0,n) \prec (1,n)$ in $L$, $(1,m-1) \wedge r = (1,m-1) \wedge (0,n) < (1,m-1) \wedge (r \vee (0,n))$. But this contradicts the fact that $L$ join semidistributive. This completes the proof.\\
\end{proof}

One application of the developments is as follows. It appears that Theorem \ref{spanning pair} and Theorem \ref{spanning pair SD} combined with A. Day's characterization of finite bounded lattices could be useful for describing finite width sublattices of free lattices. 
All sublattices of a free lattice are semidistributive and satisfy Whitman's condition; and a lattice which satisfies Whitman's condition has no doubly reducible elements. Moreover, in a finite width lattice there is a positive integer $N$ such that \emph{every} antichain has cardinality at most $N$.\\

It is hoped that the work done in this paper may lead to new ways of attacking the following open problem: \emph{Which countable lattices are isomorphic to a sublattice of a free lattice?} Concluding this paper, we quote J. Reinhold from \cite{WDLFL}: \emph{we are still far away from a characterization of arbitrary sublattices of free lattices}.

\end{document}